\documentclass[smallcondensed]{svjour3}

\usepackage{latexsym,mathrsfs}
\usepackage{amsmath,amssymb}
\usepackage{enumerate,verbatim}
\usepackage{amsfonts,cite,color}
\usepackage{float}

\newdimen\graphsize
\usepackage{epsf}
\usepackage{graphicx}
\newcommand{\IR}{{\mathbb{R}}}

\newcommand{\cS}{{\mathcal{S}}}

\numberwithin{equation}{section}

\def\R{{\rm I\!R}}
\def\Z{\rm \mathbb{Z}}
\def\cS{{\cal S}}

\def\CP{{\cal CP}}
\def\N{{\cal N}}
\def\SDD{\textup{SDD}}
\def\Diag{\textup{Diag}}
\def\DD{\textup{DD}}
\date{}

\newcommand{\revise}[1]{\textcolor{black}{\ignorespaces#1\ignorespaces}}   
\newcommand{\rerevise}[1]{\textcolor{black}{\ignorespaces#1\ignorespaces}}   

\def\CoP{{\cal COP}}
\def\C{{\cal C}}

\title{Inner approximating the completely positive cone via the cone of scaled diagonally dominant matrices}

\author{Jo\~{a}o Gouveia \and Ting Kei Pong \and Mina Saee\thanks {
The first author's research was partially supported by FCT under grants UID/MAT/00324/2019 through CMUC, and P2020 SAICTPAC/0011/2015. The second author's research was supported partly by a research grant (G-UADF) from The Hong Kong Polytechnic University and the third author's research was supported by a PhD scholarship from FCT, grant PD/BD/128060/2016.
}
}
\institute{
J. Gouveia \at CMUC, Department of Mathematics, University of Coimbra, Portugal\\
\email{jgouveia@mat.uc.pt}
\and
T. K. Pong \at
Department of Applied Mathematics, The Hong Kong Polytechnic University, Hong Kong, People's Republic of China\\
\email{tk.pong@polyu.edu.hk}
\and
M. Saee \at
Department of Mathematics, University of Coimbra, Portugal\\
\email{minasaee@mat.uc.pt}}

\begin{document}

\maketitle

\begin{abstract}
  Motivated by the expressive power of completely positive programming to encode hard optimization problems, many approximation schemes for the completely positive cone have been proposed and successfully used. Most schemes are based on outer approximations, with
the only inner approximations available being a linear programming based method proposed by Bundfuss and D\"{u}r \cite{BunDur09} and also Y\i ld\i r\i m \cite{Yil12}, and a semidefinite programming based method proposed by Lasserre \cite{Lasserre14}.
In this paper, we propose the use of the cone of nonnegative scaled diagonally dominant matrices as a natural inner approximation to the
completely positive cone. Using projections of this cone we derive new graph-based second-order cone approximation schemes for completely
positive programming, leading to both uniform and problem-dependent hierarchies. This offers a compromise between the expressive power of
semidefinite programming and the speed of linear programming based approaches. Numerical results on random problems, \rerevise{standard quadratic programs} and the stable set problem are presented to illustrate the effectiveness of our approach.
\end{abstract}

%
%
%

\section{Introduction}

Copositive programming and its dual \revise{counterpart} of completely positive programming are classes of convex optimization problems that have in the past decades developed as a particularly expressive tool to encode optimization problems, especially for many problems arising from combinatorial or quadratic optimization. A classical example of that can be found in \cite{Bur09}, which shows that general quadratic programs with a mix of binary and continuous variables can be expressed as copositive programs. A large body of work has been developed in the area and there is a series of survey papers that can be consulted for further information. We refer the readers to \cite{Bom12,Bur12,Dur10} and references therein for more details.

In this paper we will focus on general completely positive programs which are linear optimization problems of the form (see Section~\ref{sec1.1} below for notation)
\begin{equation}\label{P0}
  \begin{array}{rl}
    v_p:=\min & {\rm tr}(CX)\\
    {\rm s.t.} & {\rm tr}(A_i X) = b_i,\ i = 1,\ldots,m,\\
    & X \in \CP^n,
  \end{array}
\end{equation}
where $C$ and $A_i$, $i = 1,\ldots,m$ are symmetric matrices, and $\CP^n$ is the closed cone of \revise{$n \times n$} completely positive matrices \revise{defined as
	\begin{align}
	\CP^n& := \{X\in \cS^n:\; \exists B \ge 0, \ \ X = B^TB\}.\label{CPn}
	\end{align}}

We also consider the dual problem of \eqref{P0}, which is the following copositive programming problem
\begin{equation}\label{P0dual}
\begin{array}{rl}
v_d:= \max & b^Ty\\
{\rm s.t.} & C - \sum_{i=1}^m y_i A_i \in \CoP^n,
\end{array}
\end{equation}
where $\CoP^n$ is the closed cone of \revise{$n \times n$} copositive matrices \revise{and is defined as
	\begin{align}
	\CoP^n& := \{X\in \cS^n:\; v^TXv \ge 0,\ \ \forall v\ge 0\}.\label{COPn}
	\end{align}}

It is well known that completely positive programming problems \eqref{P0} are NP-hard in general. Several approximation schemes have been proposed and successfully used in the literature, based on approximations to $\CP^n$. The simplest one is to replace $\CP^n$ by the cone of nonnegative positive semidefinite matrices, which is strictly larger \revise{than $\CP^n$ } when $n\ge 5$, hence leading to a lower bound to $v_p$. Other popular lower bounds are those relying on semidefinite programming sums of squares techniques as introduced in \cite{Par00}. For upper bounds based on inner approximations to $\CP^n$, the literature is somewhat sparser.

One way of constructing inner approximations to $\CP^n$ is to make use of the fact that the extreme rays of $\CP^n$ are matrices of the form $vv^T$ with $v\in \R_+^n\backslash\{0\}$; \revise{see \cite{abraham2003completely}}. Thus, one can pick uniformly spaced $v \in \Delta^n=\left\{ x \in \R_+^n :\; \sum x_i = 1\right\}$, and approximate $\CP^n$ by the cone the matrices $vv^T$ generate (see \cite{Yil12,BunDur09}). This leads to linear programming (LP) approximations to \eqref{P0}. Another inner approximation to $\CP^n$ is that proposed in \cite{Lasserre14}, based on the theory of moments, leading to semidefinite programming (SDP) approximations to \eqref{P0}. In both cases we have hierarchies that give upper bounds to \eqref{P0}, and dually lower bounds to \eqref{P0dual}, and converge to the optimal value/solutions of \eqref{P0}. These inner approximations are uniform (i.e., problem-independent) approximations, giving rise to either LP or SDP problems. See also \cite{Yil16} for a more thorough treatment of inner approximations. An extra step taken as an adaptive linear approximation algorithm was proposed in \cite{BunDur09}. This uses information obtained from an upper bound approximation to selectively refine the hierarchy, leading to problem-dependent LP approximations.

In this paper, we propose a new inner approximation scheme to $\CP^n$ that is based on {\em second-order cone programming} (SOCP) problems and can be either uniform or problem-dependent. Our approach is motivated by the recent work in \cite{Ahm14,AhmHal17} that uses the cone of scaled diagonally dominant matrices for inner-approximating the cone of positive semidefinite matrices. Specifically, we use the cones of \revise{nonnegative scaled diagonally dominant matrices} and their projections as a natural inner approximation to $\CP^n$, and derive a new SOCP-based approximation scheme for completely positive and copositive programming. Our approximation scheme has a natural graphical interpretation. By exploiting this interpretation, we can flexibly expand or trim the SOCP problems in our hierarchy, leading to both uniform and problem-dependent approximation schemes. The use of SOCP offers a compromise between the expressive power of SDP, that comes at a \revise{significant computational cost}, and the speed of LP approaches, that have inherently lower expressive power. Numerical experiments on solving random instances, \rerevise{standard quadratic programs} and the stable set problem demonstrate the effectiveness of our approximation schemes.

%

The rest of the paper is organized as follows. We present notation and state our blanket assumptions concerning \eqref{P0} and \eqref{P0dual} in Section~\ref{sec1.1}. Properties of the scaled diagonally dominant matrices are reviewed in Section~\ref{sec:sdd}, and a graphical refinement scheme is discussed. We derive our uniform inner approximation schemes in Section~\ref{sec3} with a convergence analysis, and discuss several problem-dependent inner approximation schemes in Section~\ref{sec4}. Numerical experiments are reported in Section~\ref{sec:num}.

\subsection{Notation and blanket assumptions}\label{sec1.1}

In this paper, we use $\cS^n$ to denote the space of \revise{$n \times n$} symmetric matrices. Matrices are denoted by upper case letters, and their entries are represented in the corresponding lower case letters, e.g., $d_{ij}$ as the $(i,j)$th entry of the matrix $D$; we also use lower case letters to denote vectors. For vectors $u$, $v\in \R^n$, we write $u\ge 0$ if $u$ is elementwise nonnegative, and use $[u,v]$ to denote the line segment between $u$ and $v$, i.e.,
\[
[u,v] := \{tu + (1-t)v:\; t\in [0,1]\}.
\]
For an $X\in \cS^n$, we write $X\succeq 0$ if \revise{$X$} is positive semidefinite, and write $X\ge 0$ if \revise{$X$} is elementwise nonnegative. We also write the trace of $X$ as ${\rm tr}(X)$. We use $E$ and $I$ to denote the square matrix of all ones and the identity matrix, respectively, whose dimensions should be clear from the context.  Finally, for a linear map ${\cal A}:\cS^n\to \cS^m$, we use ${\cal A}^*$ to denote its adjoint.

The cone of positive semidefinite \revise{$n \times n$} matrices is denoted by $\cS^n_+$. We also use $\N^n$ to denote the cone of \revise{$n \times n$} symmetric nonnegative matrices, i.e.,
\[
\N^n:=\{X\in \cS^n:\; X\ge 0\},
\]
 \revise{ and an \revise{$n \times n$} real symmetric matrix is doubly nonnegative if it is positive semidefinite and entrywise nonnegative.}
It is known that the cones in \eqref{CPn} and \eqref{COPn} are dual to each other, i.e., $\CP^n = (\CoP^n)^*$ and $\CoP^n = (\CP^n)^*$; here,
\[
{\C}^* := \{Y\in \cS^n:\; {\rm tr}(XY)\ge 0,\ \ \forall X\in {\C}\}
\]
for a closed convex cone ${\C}\subseteq \cS^n$. Moreover, it is also known that the cone of positive semidefinite matrices and the cone of symmetric nonnegative matrices are self-dual, i.e., $\cS^n_+ =(\cS^n_+)^*$ and $\N^n=(\N^n)^*$.

Throughout this paper, we make the following {\bf blanket assumptions} concerning \eqref{P0} and \eqref{P0dual}:
\begin{enumerate}[{\bf A}1.]
  \item Problem \eqref{P0} is feasible.
  \item The mapping $X\mapsto ({\rm tr}(A_1X),\ldots,{\rm tr}(A_mX))$ is surjective.
  \item Problem \eqref{P0dual} is strictly feasible, i.e., there exists $\bar y$ satisfying
  \[
  C - \sum_{i=1}^m \bar y_i A_i \in {\rm int}\,\CoP^n.
  \]
\end{enumerate}
Under these assumptions, the dual \revise{Slater} condition holds. Therefore we have $v_p = v_d$, with both values being finite and the primal optimal value $v_p$ being attained.

\section{The scaled diagonally dominant cone and beyond}\label{sec:sdd}

In this section, we present the basis for our construction of inner approximations in Sections~\ref{sec3} and \ref{sec4}. Our construction is motivated by the work in \cite{Ahm14,AhmHal17}, which studied inner approximations of the cone of positive semidefinite matrices based on the cones of diagonally dominant and scaled diagonally dominant matrices. While their work can be directly applied to the existing SOS hierarchies to yield outer approximations of $\CP^n$ (see \cite[Section 4.2]{Ahm14}) we show an alternative approach, based on the same cones but using them in a fundamentally different way, in order to obtain an inner approximation to $\CP^n$. We first recall the following definition of diagonally dominant and scaled diagonally dominant matrices from \revise{\cite[Definition~3.3]{Ahm14}}.
\begin{definition}\label{def:sdd}
A symmetric matrix $A$ is diagonally dominant\footnote{Note that our definition is different from the classical definition of diagonal dominance (see \cite[Definition~6.1.9]{HornJohnson85}) in that we require the diagonal entries of $A$ to be nonnegative.} if $a_{ii}\ge \sum_{j\neq i}|a_{ij}|$ for all $i$, and is said to be scaled diagonally dominant (sdd) if there exists a diagonal matrix $D$ with positive diagonal entries such that $DAD$ is diagonally dominant.
\end{definition}

In \cite[Theorem~9]{boman2005factor} a convenient characterization of sdd matrices was presented. They proved that a matrix is sdd if and only if it can be written as the sum of positive semidefinite matrices whose supports are contained in some \revise{$2 \times 2$} submatrices. In other words, the cone $\SDD^n$ of \revise{$n \times n$} sdd matrices is given by
\begin{equation}\label{sdd:eqdef}
\SDD^n := \sum_{1\le i<j\le n} \iota_{ij}(\cS^2_+),
\end{equation}
where $\iota_{ij}:\cS^2 \to \cS^n$ is the map that sends an $S\in \cS^2$ to the matrix $D$ given by
\[
d_{rs}:= \begin{cases}
  s_{11} & {\rm if}\ (r,s) = (i,i),\\
  s_{12} & {\rm if}\ (r,s) = (i,j),\\
  s_{21} & {\rm if}\ (r,s) = (j,i),\\
  s_{22} & {\rm if}\ (r,s) = (j,j),\\
  0 & {\rm otherwise}.
\end{cases}
\]
This cone is therefore given in terms of \revise{$2 \times 2$} semidefinite constraints or, in other words, second-order cone constraints, which makes it quite suitable to use in convex optimization.
One can prove the following basic properties of $\SDD^n$, and of the set $\SDD^n_+:= \SDD^n\cap \N^n$. Note that item (i) in Proposition~\ref{prop1} below can be found in \cite{AhmHal17}, and a more general version of it can be found in \cite[Lemma~5]{Permenter2017}. We include it here for completeness. \revise{In what follows $\iota_{ij}^*$ denotes the adjoint of the map $\iota_{ij}$, which in this case can be defined by saying that $\iota_{ij}^*(S)$ is the $2\times 2$ submatrix of $S$ indexed by rows and columns $i$ and $j$.}

\begin{proposition}\label{prop1}
  The following statements hold.
  \begin{enumerate}[{\rm (i)}]
    \item $(\SDD^n)^* = \{Q\in \cS^n:\; \iota_{ij}^*(Q)\succeq 0,\ \ \forall 1\le i< j\le n\}$.
    \item $(\SDD^n_+)^* = (\SDD^n)^* + \N^n$.
    \item $
  \SDD^n_+ = \sum_{1\le i<j\le n} \iota_{ij}(\cS^2_+\cap \N^2)$.
  \end{enumerate}
\end{proposition}

\begin{proof}
  We first prove (i). Recall from \eqref{sdd:eqdef} that $\SDD^n = \sum_{1\le i<j\le n}\iota_{ij}(\cS^2_+)$. Thus, we have from \cite[Corollary~16.3.2]{Rock70} that
  \[
  (\SDD^n)^* = \bigcap_{1\le i<j\le n}(\iota_{ij}(\cS^2_+))^*,
  \]
  from which the desired equality follows immediately.

  Next, we prove (ii). Note that
  \[
  \sum_{1\le i<j\le n}\iota_{ij}(E)\in \SDD^n\cap {\rm int}\,\N^n.
  \]
  Thus, we conclude from \cite[Corollary~16.3.2]{Rock70} that
  \[
  (\SDD^n_+)^* = (\SDD^n\cap \N^n)^* = (\SDD^n)^* + \N^n.
  \]

  Finally, we prove (iii). It is clear that $\SDD^n_+ \supseteq \sum_{1\le i<j\le n} \iota_{ij}(\cS^2_+\cap \N^2)$. For the converse inclusion, consider any $Q\in \SDD^n_+$. Then $Q$ is nonnegative and can be written as $\sum_{1\le i<j\le n} \iota_{ij}(S_{ij})$ for some $S_{ij}\in \cS^2_+$, $1\le i<j\le n$. Observe that each $S_{ij}$ has nonnegative diagonal entries, and moreover, its nondiagonal entry equals the $(i,j)$th entry of $Q$, which is also nonnegative. Thus, $S_{ij}\in \cS^2_+\cap \N^2$ and hence $Q\in \sum_{1\le i<j\le n} \iota_{ij}(\cS^2_+\cap \N^2)$.
  This completes the proof.
\end{proof}

Since \revise{$2 \times 2$} nonnegative positive semidefinite matrices are completely positive, we see from Proposition~\ref{prop1}(iii) that $\SDD^n_+$ is an inner approximation to $\CP^n$. In Figure \ref{fig:comparison} we show a random $2$-dimensional slice of the cone of doubly nonnegative \revise{$5 \times 5$} matrices (i.e., $\cS_+^5\cap \N^5$) with the slice of  $\SDD^5_+$ highlighted in red. The cone $\CP^5$ is sandwiched between them.
\begin{figure}
\begin{center}
\includegraphics[width=5.5cm]{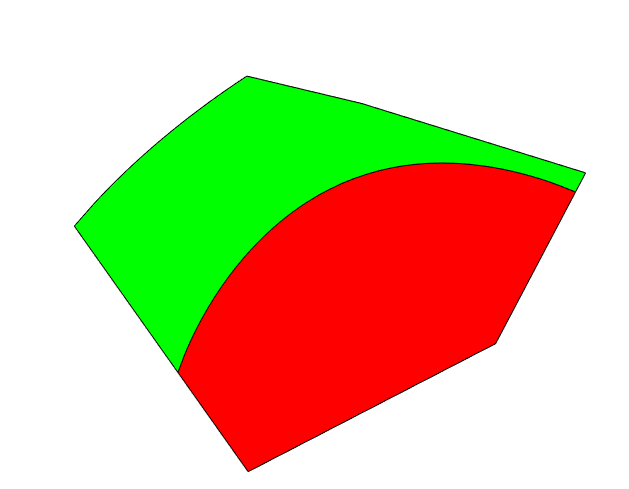}
\end{center}
\caption{Comparison of $\cS^5_+\cap \N^5$ with $\SDD^5_+$}
\label{fig:comparison}
\end{figure}

This simple inner approximation can be used as a basis to construct more general inner second-order cone approximations for $\CP^n$. To do that we consider a useful variant of $\SDD^n_+$ that will help us construct inner approximations of $\CP^n$.

\begin{definition}
Let $U\in \R^{t\times n}_+$ have \revise{row sum $1$}. Define
\begin{equation}\label{SDDplusU}
  \SDD^n_+(U) := \{U^TYU:\; Y \in \SDD^t_+\} =U^T (\SDD^t_+)U.
\end{equation}
\end{definition}
The above definition is similar to the development in \cite[Section~3.1]{AhmHal17}, which makes use of the so-called $\DD(U)$. Here we assume that $U$ has nonnegative entries so that $\SDD^n_+(U)$ will be a subcone of $\CP^n$; see Proposition~\ref{prop2} below. In addition, we assume that the rows of $U$ have sum one: we can then always think of the rows of $U$ as points in the simplex $\Delta^n$. This is no less general than just considering $U\in \R^{t\times n}_+$ with nonzero rows, because scaling rows of $U$ by positive scalars does not change $\SDD^n_+(U)$. \revise{Note that $\SDD^n_+(U)$ is simply a linear image of $\SDD^t_+$ into $\cS^n$. }

\revise{Some basic properties of this set are that $\SDD^n_+(I_n)=\SDD^n_+$, and that if $U\in \R^{t\times n}_+$ is a submatrix of $\tilde{U} \in \R^{s\times n}_+$ then $\SDD^n_+(U) \subseteq \SDD^n_+(\tilde{U})$. We show in the next example that $\SDD^n_+(U)$ can be strictly larger than $\SDD^n_+$ in general.}

\begin{example}\revise{
One can see that the matrix
$$M=\begin{bmatrix} 6 & 5 & 5 \\ 5 & 6 & 5 \\ 5 & 5 & 6 \end{bmatrix}$$
is in $\cS_+^3 \cap \N^3$. However, $M\notin \SDD^3_+$; indeed, if we define
\[
W:= \begin{bmatrix}
  1&-1&-1\\-1&1&-1\\-1&-1&1
\end{bmatrix},
\]
then ${\rm tr}(WM) < 0$ but $W\in (\SDD_+^3)^*$ thanks to Proposition~\ref{prop1}(ii), showing that $M\notin \SDD^3_+$.}

\revise{Now, suppose we set $U$ to be the $4 \times 3$ matrix
constructed from concatenating the identity $I_3$ with an all $\frac13$ row vector, i.e.,
\[
U = \begin{bmatrix}
  1&0&0\\
  0&1&0\\
  0&0&1\\
  \frac13&\frac13&\frac13
\end{bmatrix},
\]
and consider the set $\SDD^3_+(U)$. Then we know $\SDD^3_+ \subseteq \SDD^3_+(U)$ because $I_3$ is a submatrix of $U$. Furthermore, we have
\[
M=U^T \begin{bmatrix} 1 & 0 & 0 & 3 \\ 0 & 1 & 0 & 3 \\ 0 & 0 & 1 & 3 \\ 3 & 3 & 3& 27 \end{bmatrix} U \in \SDD^3_+(U),
\]
where the inclusion holds because
\[
\begin{bmatrix} 1 & 0 & 0 & 3 \\ 0 & 1 & 0 & 3 \\ 0 & 0 & 1 & 3 \\ 3 & 3 & 3& 27 \end{bmatrix}
= \begin{bmatrix} 1 & 0 & 0 & 3 \\ 0 & 0 & 0 & 0 \\ 0 & 0 & 0 & 0 \\ 3 & 0 & 0& 9 \end{bmatrix}
+\begin{bmatrix} 0 & 0 & 0 & 0 \\ 0 & 1 & 0 & 3 \\ 0 & 0 & 0 & 0 \\ 0 & 3 & 0& 9 \end{bmatrix}
+\begin{bmatrix} 0 & 0 & 0 & 0 \\ 0 & 0 & 0 & 0 \\ 0 & 0 & 1 & 3 \\ 0 & 0 & 3& 9 \end{bmatrix}\in \SDD^4_+.
\]
Consequently, $\SDD^3_+(U)$ is a strictly larger set than $\SDD^3_+$.}
\end{example}

We next give an important characterization of $\SDD^n_+(U)$ that is crucial in our development of inner approximation schemes in Sections~\ref{sec3} and \ref{sec4}. Recall from \eqref{CPn} that $\CP^n$ can be seen as the convex hull of all $vv^T$ with $v \in \R_+^n$. The next theorem shows that one can think of $\SDD^n_+(U)$ similarly.

\begin{theorem}\label{prop:sddUconehull}
Let $U\in \R^{t\times n}_+$ have \revise{row sum $1$}. Then $\SDD^n_+(U)$ is the conic hull of all $vv^T$ with $v$ belonging to some line segment $[u_i,u_j]$, where $u_i$ is the $i$-th row of $U$.
\end{theorem}
\begin{proof}
\revise{Note from Proposition~\ref{prop1}(iii) and \eqref{SDDplusU} that any matrix in $\SDD^n_+(U)$ can be written as
$$\rerevise{\sum_{1\leq i< j \leq t}U^T \iota_{ij}(S_{ij}) U}$$ for some $S_{ij} \in \cS_+^2\cap \N^2$.
Moreover, any matrix $S\in \cS_+^2\cap \N^2$ can be written as $S=v_1v_1^T+v_2v_2^T$ for some nonnegative vectors $v_i \in \IR_+^2$. Furthermore, we know that for any $v\in \IR^2$, it holds that $\iota_{ij}(vv^T)=ww^T$ where $w\in \mathbb{R}^t$ is the vector whose $i$th entry is $v_1$, $j$th entry equals $v_2$, and is zero otherwise. Hence, we deduce that any matrix in $\SDD^n_+(U)$ can be written as
$$\sum_{k=1}^N U^Tw_k w_k^TU,$$
where each $w_k\in \mathbb{R}^t$ is nonnegative and has a support of cardinality at most $2$. Conversely, it is easy to see that any matrix that can be written as such a sum is in $\SDD^n_+(U)$. But each $U^Tw$, with $w\neq 0$, is simply a (nonzero) conic combination of two rows of $U$, $u_i$ and $u_j$; so, up to positive scaling, it is in $[u_i,u_j]$, proving our claim.}
\end{proof}

We can now prove the following properties of $\SDD^n_+(U)$.

\begin{proposition}\label{prop2}
  Let $U\in \R^{t\times n}_+$ have \revise{row sum $1$}. Then the following statements hold.
  \begin{enumerate}[{\rm (i)}]
    \item The cone $\SDD^n_+(U)$ is a closed sub-cone of $\CP^n$.
    \item $(\SDD^n_+(U))^* = \{Y:\; UYU^T\in (\SDD^t_+)^*\} = \{Y:\; UYU^T\in (\SDD^t)^* + \N^t\}$.
  \end{enumerate}
\end{proposition}

\begin{proof}
  From Theorem~\ref{prop:sddUconehull}, it follows that  $\SDD^n_+(U)$ is a sub-cone of $\CP^n$. It remains to prove closedness. Since $U$ is nonnegative and has no zero rows, the origin is not in the convex hull of $vv^T$, where $v$ belongs to some $[u_i,u_j]$, and $u_i$ is the $i$-th row of $U$. Hence $\SDD^n_+(U)$ is the conic hull of a compact convex set not containing the origin. Thus, it is closed.

  To prove (ii), recall that $\SDD^n_+(U) =U^T (\SDD^t_+)U$. From this we see that $Y\in (\SDD^n_+(U))^*$ if and only if
  \[
  {\rm tr}(Y(U^TWU))\ge 0\ \ \ \forall W\in \SDD^t_+,
  \]
  which is the same as $UYU^T\in (\SDD^t_+)^*$. This proves the first equality. The second equality in (ii) follows from Proposition~\ref{prop1}(ii). This completes the proof.
\end{proof}

Note that the construction of $\SDD^n_+(U)$ is fairly general. Anytime we have a cone $\C\subseteq \CP^t$ and a matrix $U \in \R_+^{t \times n}$ whose rows have sum one, one can define the cone
\begin{equation}\label{CU}
  \C(U) := \{U^TYU:\; Y \in {\C}\} =U^T {\C}U.
\end{equation}
This is easily seen to always verify $\C(U) \subseteq \CP^n$, since $\C(U) \subseteq U^T\CP^tU \subseteq\CP^n$. It is helpful to state in this language the usual LP inner approximations to $\CP^n$. Let $\Diag_+^n$ be the set of nonnegative \revise{$n \times n$} diagonal matrices. Clearly $\Diag_+^n \subseteq \CP^n$, so we can define
\begin{equation}\label{DiagplusU}
  \Diag_+^n(U) := \{U^TYU:\; Y \in \Diag_+^t\}.
\end{equation}
This is nothing more than the conic hull of the matrices $u_iu_i^T$, $i=1,\dots,t$, where $u_i$ is the $i$-th row of $U$. The use of \eqref{DiagplusU} for inner approximation corresponds to the standard LP approximation strategy used, for example, in \cite{BunDur09}, where strategies for efficient choices of $U$ were explored.

Another possibility for obtaining an LP relaxation would be to use the cone of \revise{$n \times n$} symmetric nonnegative diagonally dominant matrices, denoted by $\DD^n_+$. We have $\Diag_+^n \subseteq \DD_+^n \subseteq \SDD_+^n$. So, if we define
\begin{equation}\label{DDplusU}
  \DD_+^n(U) := \{U^TYU:\; Y \in \DD_+^t\},
\end{equation}
we would get $\Diag_+^n(U) \subseteq \DD_+^n(U) \subseteq \SDD_+^n(U)$. However, since one can easily see that $\DD_+^n$ is the conic hull of $(e_i+e_j)(e_i+e_j)^T$ for $1 \leq i \leq j \leq n$, it is not hard to see that $\DD_+^n(U)$ is simply the conic hull of $(u_i+u_j)(u_i+u_j)^T$ for $1 \leq i \leq j \leq t$, and hence can be expressed in terms of $\Diag_+^n(U')$ for some $U'$ that contains $U$ as a submatrix.

Other choices would be to use not submatrices in $\cS_+^2$, as we did for $\SDD^n_+$, but matrices in $\cS_+^3$ or $\cS_+^4$. Note that it is still true in these two cases that $\cS_+^i \cap \N^i\subseteq \CP^i$. 
These cones would give better approximations, but we would get a much higher number of constraints that would not be second-order cone constraints but fully semidefinite. \revise{While the semidefinite constraints would still be small, the process would become more cumbersome and significantly less tractable}.

\subsection{A graphical refinement}

We saw above that $\SDD^n_+(U)$ is a natural inner approximation to $\CP^n$. Furthermore, Theorem~\ref{prop:sddUconehull} suggests that the fundamental property of $U$ that guides the approximation is the collection of segments $[u_i,u_j]$. \revise{We might associate to the points $u_i$ vertices of a graph, and to the segments its edges, and think of the collection of points and segments as a concrete realization of the graph in $\R^n$. This insight can be used to refine the approximation, making it more flexible.}
We start by generalizing the notion of $\SDD$.

Given a graph $G$ with vertex set $\{1,\dots,n\}$ and edge set ${\cal E}$, we define
\[
\SDD^G := \sum_{\{i,j\} \in {\cal E}} \iota_{ij}(\cS^2_+),
\]
\revise{and we set $\SDD^G := \{0\}$ if ${\cal E} = \emptyset$ by convention. The graph $G$ simply encodes which principal $2\times 2$ submatrices will be required to be semidefinite. In particular,} if we consider $G$ to be the complete graph $K^n$, this is simply $\SDD^n$. We can define $\SDD_+^G$ as the nonnegative matrices in $\SDD^G$, similarly as before. Then we can naturally define a generalization of $\SDD_+^n(U)$\revise{:}

\begin{definition}
\revise{For a graph $G$ with $t$ vertices and a matrix $U \in \R_+^{t \times n}$ whose rows have sum one,} we define the cone $\SDD^G_+(U)$ as
 $$ \SDD^G_+(U) := \{U^TYU:\; Y \in \SDD^G_+\} =U^T (\SDD^G_+)U.$$
\end{definition}

\revise{It will be helpful to think of the rows of $U$ as points in the standard simplex $\Delta^n$ (i.e. with nonnegative coordinates summing to one). These points correspond to vertices of the graph $G$, and the edge set of $G$ simply encodes which pairs of rows of $U$ (vertices) are ``connected".} In other words, the pair $(G,U)$ is a realization of the graph $G$ inside \revise{ $\Delta^n$  with segments for edges}. We will denote by ${\rm seg}(G,U)$ the set of points in some \revise{ of the segments}, i.e,
$${\rm seg}(G,U)= \bigcup_{\{i,j\} \in {\cal E}} [u_i,u_j],$$
where $u_i$ is the $i$-th row of $U$. \revise{ This set completely controls the geometry of the cone. Based on this notion and the proof of Theorem~\ref{prop:sddUconehull}, we can immediately obtain the following refinement of Theorem~\ref{prop:sddUconehull} for the representation of $\SDD^G_+(U)$.}
\revise{\begin{theorem}\label{prop:sddGUconehull}
Let $G$ be a graph with $t$ vertices and $U \in \R_+^{t \times n}$ be a matrix whose rows have sum one. Then $\SDD^G_+(U)$ is the conic hull of all $vv^T$ with $v \in {\rm seg}(G,U)$.
\end{theorem}}
\revise{Theorem~\ref{prop:sddGUconehull} gives a simple way of translating results from the graph language to results about cones. In particular if we have ${\rm seg}(G,U) \subseteq {\rm seg}(G',U')$, we have $\SDD^G_+(U) \subseteq \SDD^{G'}_+(U')$, and
furthermore $\SDD^G_+(U) \subseteq \SDD^{K^t}_+(U)=\SDD^n_+(U) \subseteq \CP^n$, for all graphs $G$ with $t$ vertices and matrices $U \in \R_+^{t \times n}$ whose rows have sum one.} \revise{On the other hand, if every node of the graph $G$ is covered by some edges, then $\SDD^G_+(U) \supseteq \Diag_+^n(U)$, the usual LP inner approximation. Thus, the graphical notation allows us to construct intermediate approximations somewhere in between the simple LP inner approximation and the full $\SDD^n_+(U)$ version.}

We end the section by noting that most of our other previous results concerning $\SDD^n_+$ and $\SDD^n_+(U)$ can be adapted with no effort to this new cone.
\begin{theorem}\label{prop:properties}
Given a graph $G$ with \revise{$t$} vertices and edge set ${\cal E}$, and a matrix $U \in \R_+^{t \times n}$ whose rows have sum one, we have the following properties.
\begin{enumerate}[{\rm (i)}]
 \item  $(\SDD^G)^* = \{Q\in \cS^n:\; \iota_{ij}^*(Q)\succeq 0\ \ \forall \{i,j\} \in {\cal E} \}$;
 \item  $\SDD^G_+ = \sum_{\{i,j\}\in {\cal E}}\iota_{ij}(S^2_+\cap \N^2)$;
 \item  $(\SDD^G_+(U))^* = \{Y:\; UYU^T\in (\SDD^G_+)^*\}$;
 \item  $\SDD^G_+(U)$ is a closed sub-cone of $\CP^n$.
 \end{enumerate}
\end{theorem}
\begin{proof}
\revise{Immediate from the proofs of Proposition~\ref{prop1} and Proposition~\ref{prop2}.} 
\end{proof}

\section{Inner approximation schemes for the completely positive cone}\label{sec3}

The main idea of this section is to approximate the solution to \eqref{P0} by using the cones  $\SDD^G_+(U)$ to replace $\CP^n$. More concretely
our scheme is based on the following family of optimization problems, which depends on a graph $G$ on \revise{$t$} vertices and a $U\in \R_+^{t\times n}$ whose rows have sum one:
\begin{equation}\label{SDDP0}
  \begin{array}{rl}
    v_p(G,U):=\min & {\rm tr}(CX)\\
    {\rm s.t.}& {\rm tr}(A_iX) = b_i,\  i =1,\ldots,m,\\
    & X\in \SDD^G_+(U),
  \end{array}
\end{equation}
and its dual problem given by
\begin{equation}\label{SDDP0dual}
  \begin{array}{rl}
    v_d(G,U):=\max & b^Ty\\
    {\rm s.t.}& C - \sum_{i=1}^m y_i A_i\in (\SDD^G_+(U))^*.
  \end{array}
\end{equation}
Note that the semidefinite constraints in \eqref{SDDP0} are imposed only on $2\times 2$ matrices. Thus, these problems are SOCP problems.

Recall from Theorem~\ref{prop:properties} that $\SDD^G_+(U)$ and $(\SDD^G_+(U))^*$ are both closed convex cones. Also, notice that \eqref{SDDP0dual} has a strictly feasible point due to \revise{Assumption {\bf A}3} and the fact that $\CoP^n\subseteq (\SDD^G_+(U))^*$ (which follows from $\SDD^G_+(U) \subseteq \CP^n$).
Consequently, if Problem \eqref{SDDP0} is feasible, then $v_p(G,U) = v_d(G,U)$, both values are finite and $v_p(G,U)$ is attained. Moreover, we conclude from $\SDD^G_+(U) \subseteq \CP^n$ that $v_p(G,U)\ge v_p$.
Furthermore, we have already pointed out that augmenting the embedded graph $(G,U)$ leads to an enlargement in $\SDD^G_+(U)$. In view of these observations, we will discuss strategies for constructing an ``enlarging" sequence of graphs $\{(G^k,U^k)\}$ to possibly tighten the gap $v_p(G^k,U^k) - v_p$ as $k$ increases.

To simplify our terminology, we make the following definition.

\begin{definition}
  A sequence of embedded graphs $\{(G^k,U^k)\}$ is called a positively enlarging sequence if ${\rm seg}(G^k,U^k) \subseteq {\rm seg}(G^{k+1},U^{k+1})$, each $U$ is a nonnegative matrix having at least $n$ rows, each row of $U$ (the realizations of vertices of $G$) sums to one, and each node of $G$ is covered by at least one edge.
\end{definition}

Positively enlarging sequences verify $v_p(G^k,U^k)\ge v_p(G^{k+1},U^{k+1})\ge v_p$ by construction. Furthermore, once \eqref{SDDP0} is feasible for some $k=k_0$, it will remain feasible whenever $k\ge k_0$, since the sequence of sets $\{\SDD^{G_k}_+(U^k)\}$ are monotonically increasing. Moreover, we have noted above that we might think of the rows of $U$ to be in the simplex $\Delta^n$ so that we can think of this as an enlarging family of graphs embedded in $\Delta^n$.

We next study convergence of our inner approximation schemes for \eqref{P0} based on \eqref{SDDP0} when $\{(G^k,U^k)\}$ is a positively enlarging sequence. We first prove a convergence result concerning a similar approximation scheme, which uses ${\Diag}_+^n(U)$ (as defined in \eqref{DiagplusU}) in place of $\SDD^G_+(U)$ in \eqref{SDDP0}. This strategy was used in \cite{BunDur09}, which studied the pairs \eqref{SDDP0} and \eqref{SDDP0dual} with ${\Diag}_+^n(U)$ in place of $\SDD^G_+(U)$, and constructed an ``enlarging" sequence $\{U^k\}$ by adding new rows to $U^k$ from $\Delta^n$ at each step. To determine what rows to add, they solve another LP approximation scheme based on $U$, which they see as the set of vertices of a simplicial partition of $\Delta^n$, 
and use its results to construct a sequence of $\{U^k\}$ with an increasing number of rows. In studying the convergence of that method they proved a version of the following result for copositive programming problems in \cite[Theorem~4.2]{BunDur09}. The version presented below will be useful for studying convergence of our inner approximation schemes for \eqref{P0}.

\begin{theorem}\label{convergdiag}
Assume that \eqref{P0} is strictly feasible.
Let $\{U^k\}$ be a \revise{sequence of matrices} whose rows have sum one, where for each $k$, $U^k\in \R_+^{t_k\times n}$ for some $t_k\ge n$. Suppose that
\begin{equation}\label{limit}
\lim_{k\rightarrow\infty} \max_{x \in \Delta^n} \min_{i=1,\dots,t_k} \|x-u_{i}^k\| = 0,
\end{equation}
where $u_i^k$ is the $i$-th row of $U^k$. Consider for each $k$ the following problem:
\begin{equation}\label{Diag0}
  \begin{array}{rl}
    \tilde v_{p}(U^k):=\min & {\rm tr}(CX)\\
    {\rm s.t.}& {\rm tr}(A_iX) = b_i,\  i =1,\ldots,m,\\
    & X\in {\Diag}_+^n(U^k).
  \end{array}
\end{equation}
Then the following statements hold.
\begin{enumerate}[{\rm (i)}]
  \item $\tilde v_p(U^k)$ is finite for all sufficiently large $k$ and $\lim_{k\to\infty}\tilde v_p(U^k) = v_p$.
  \item The solution set of \eqref{Diag0} is nonempty and uniformly bounded for all sufficiently large $k$.
  \item Let $X^k$ be a solution of \eqref{Diag0} whenever the solution set is nonempty. Then any accumulation point of $\{X^k\}$ is a solution of \eqref{P0}.
\end{enumerate}
\end{theorem}

\begin{proof}
  Note that the ${\Diag}_+^n(U^k)$ defined in \eqref{DiagplusU} is the conic hull of $u^k_i{u^k_i}^T$, where $u^k_i$ are rows of $U^k$. Note also that any element $X$ in $\CP^n$ can be written as the conic combination of $\frac{n(n+1)}2$ matrices $vv^T$, with $v\in \Delta^n$. Thus, in view of \eqref{limit}, $X$ can then be written as the limit of a sequence $\{X^k\}$, where $X^k\in {\Diag}_+^n(U^k)$ for each $k$. This together with ${\Diag}_+^n(U^k)\subseteq \CP^n$ shows that the sequence of sets $\{{\Diag}_+^n(U^k)\}$ converges to $\CP^n$ in the sense of Painlev\'{e}-Kuratowski \cite[Chapter~4B]{RockWets98}.

  Since the mapping $X\mapsto {\cal A}(X) := ({\rm tr}(A_1X),\ldots,{\rm tr}(A_mX))$ is surjective by Assumption~{\bf A}2 and \eqref{P0} is strictly feasible, the vector $b$ and the set ${\cal A}(\CP^n)$ cannot be separated in the sense of \cite[Theorem~2.39]{RockWets98}. Thus, \cite[Theorem~4.32]{RockWets98} shows that the sequence of feasible sets of \eqref{Diag0} converges to the feasible set of \eqref{P0} in the sense of Painlev\'{e}-Kuratowski.

  It now follows from \cite[Theorem~4.10(a)]{RockWets98} and the nonemptiness of the feasible set of \eqref{P0} that the feasible sets of \eqref{Diag0} are nonempty for all sufficiently large $k$. Hence $\tilde v_p(U^k)<\infty$ for all sufficiently large $k$. Note that for each $k$, the dual problem to \eqref{Diag0} is dual strictly feasible because of Assumption~{\bf A}3 and $\CoP^n\subseteq (\Diag_+^n(U^k))^*$. Thus, $\tilde v_p(U^k)$ is indeed finite for all sufficiently large $k$. Moreover, thanks to the dual strict feasibility, the solution sets of \eqref{Diag0} are nonempty whenever $\tilde v_p(U^k)$ is finite hence, in particular, are nonempty for all sufficiently large $k$.

  Next, note that by Assumption~{\bf A}3 the dual problems of \eqref{Diag0} for each $k$ actually have a {\em common} Slater point, i.e., there exists a matrix
  \[
  \bar Y := C - \sum_{i=1}^m\bar y _i A_i \in {\rm int}\, \CoP^n \subseteq {\rm int}\,(\Diag_+^n(U^k))^*.
  \]
  Therefore, there exists $\epsilon>0$ so that $\bar Y + \epsilon {\bf B}\subseteq {\rm int}\, \CoP^n$, where ${\bf B}$ is the unit closed ball centered at the origin (in Fr\"{o}benius norm). Consequently, for any $X\in \CP^n$, it holds that ${\rm tr}(\bar Y X)\ge \epsilon \|X\|_F$. We now argue that the solution sets of \eqref{Diag0} are uniformly bounded for all $k$. Indeed, fix any $k$ so that the solution set of \eqref{Diag0} is nonempty, and let $X^k$ be a solution. Then $X^k$ is a Lagrange multiplier for the dual problem. In particular,
  \[
  \tilde v_p(U^k) = \max_y\left\{ b^Ty + {\rm tr}\left(X^k\left[C - \sum_{i=1}^m y_i A_i\right]\right)\right\}\ge b^T\bar y + {\rm tr}(X^k \bar Y)\ge b^T\bar y + \epsilon\|X^k\|_F,
  \]
  where the last inequality holds because $X^k \in \Diag_+^n(U^k)\subseteq \CP^n$. Since $\{\tilde v_p(U^k)\}$ is nonincreasing, we conclude from the above inequality that $\{X^k\}$ can be bounded above by a constant independent of $k$. Thus, the solution sets of \eqref{Diag0} are uniformly bounded for all $k$.

  Finally, since the sequence of sets $\{\Diag_+^n(U^k)\}$ is monotonically increasing, we see from \cite[Proposition~7.4(c)]{RockWets98} that the objective function (with the constraint considered as the indicator function) of \eqref{Diag0} epi-converges to that of \eqref{P0} in the sense of \cite[Definition~7.1]{RockWets98}. The desired conclusion concerning limits of $\{\tilde v_p(U^k)\}$ and $\{X^k\}$ now follows from \cite[Theorem~7.31(b)]{RockWets98}.
\end{proof}

Since ${\Diag}_+^n(U)\subseteq\SDD^G_+(U)$ if the edges of $G$ cover all nodes, we get the convergence of the sequence of problems \eqref{SDDP0} for a positively enlarging sequence $\{(G^k,U^k)\}$ under the same assumptions on $U^k$. But we can actually \revise{obtain the desired convergence result under a weaker condition}.

\begin{theorem}\label{thm3}
Assume that \eqref{P0} is strictly feasible.
Let $\{(G^k,U^k)\}$ be a positively enlarging sequence such that
\begin{equation}\label{pt_cond}
\lim_{k\rightarrow\infty} \max_{x \in \Delta^n} \min_{y \in \revise{{\rm seg}(G^k,U^k)}} \|x-y\| = 0.
\end{equation}
Then it holds that:
\begin{enumerate}[{\rm (i)}]
  \item $v_p(G^k,U^k)$ is finite for all sufficiently large $k$ and $\lim_{k\to\infty}v_p(G^k,U^k) = v_p$.
  \item The solution set of \eqref{SDDP0} with $(G,U) = (G^k,U^k)$ is nonempty and uniformly bounded for all sufficiently large $k$.
  \item Let $X^k$ be a solution of \eqref{SDDP0} with $(G,U) = (G^k,U^k)$ whenever the solution set is nonempty. Then any accumulation point of $\{X^k\}$ is a solution of \eqref{P0}.
\end{enumerate}
\end{theorem}
\begin{proof}
Note that from Theorem~\ref{prop:sddGUconehull} and the description of $\Diag^n_+(U)$ as the conic hull of all matrices $u_iu_i^T$ where $u_i$ is a row of $U$, \revise{if every node of $G$ is covered by some edges and }if we construct $U'$ by adding rows such that each new row lies in $[u_i,u_j]$ for some $\{i,j\} \in {\cal E}$, we have $\Diag^n_+(U') \subseteq \SDD_+^G(U)$.

For each $U^k$, subdivide each segment $[u^k_i,u^k_j]$ into segments no longer than $1/k$, and \revise{add} these new points to $U^k$ to form $\tilde U^k\in \R_+^{\tilde t_k\times n}$. Then for each $x\in \Delta^n$, we have
\[
\min_{i=1,\ldots,\tilde t_k}\|x-\tilde u^k_i\| \le \min_{y \in {\rm seg}(G^k,U^k)} \|x-y\| + \frac1k,
\]
where $\tilde u_i^k$ is the $i$-th row of $\tilde U^k$. Thus, the sequence $\{\tilde U^k\}$ satisfies the conditions of Theorem~\ref{convergdiag}. Consequently, from the proof of Theorem~\ref{convergdiag}, the sequence of sets $\{{\Diag}_+^n(\tilde U^k)\}$ converges to $\CP^n$ in the sense of Painlev\'{e}-Kuratowski. In view of this and \cite[Exercise~4.3(c)]{RockWets98}, $\{\SDD^n_+(U^k)\}$ converges to $\CP^n$. The rest of the proof follows exactly the same arguments as in the proof of Theorem~\ref{convergdiag}.
\end{proof}

An obvious way of guaranteeing the satisfaction of the condition \eqref{pt_cond} in Theorem~\ref{thm3} is to consider the rows of $U^k$ to be the set of points in $x \in \Delta^n$ such that $kx \in {\Z}^n$, i.e. an equally spaced distribution of points in the simplex, with a growing number of points. This is in fact the strategy explored in \cite{Yil12} with the linear programming approach. As guaranteed by Theorem~\ref{thm3}, this is sufficient to get convergence in our case, independently of the edges considered, but we can get away with much less. Indeed, it is easy to see, for example, that we do not need to map vertices to the interior of the simplex to get convergence and, in fact, it is enough to uniformly sample the {\em boundary} of the simplex, \revise{ and form a graph with all possible edges between the chosen vertices.}
 Finding embedded graphs that optimally cover $\Delta^n$ in the sense of minimizing the maximum distance to a point of the simplex seems to be a hard problem with no obvious answer, but many different strategies can be attempted. For practical purposes, it might be helpful to use the problem structure to design strategies for constructing $\{(G^k,U^k)\}$; these may not satisfy condition \eqref{pt_cond} and hence the convergence behavior can be compromised, but their corresponding problem \eqref{SDDP0} may be easier to solve. Indeed, as discussed in \cite[Section~1.4]{LoboVandenBoydLe98}, the amount of work per iteration for solving \eqref{SDDP0dual} is ${\cal O}((m + t_k^2)^2(4|{\cal E}|+t_k^2))$ when $(G,U) = (G^k,U^k)$. Hence, we will explore some problem-dependent inner approximation schemes in the next section.

Before ending this section, we would like to point out that the approach in \cite{Yil12} using $\Diag_+^n(U)$ for (rows of) $U$ equally distributed in the simplex is one of the few problem-independent inner approximations to $\CP^n$ presented in the literature. The only other approach is that of \cite{Lasserre14}, which leads to SDP problems. Although conceptually very interesting and with guaranteed convergence, this latter approach performs poorly in practice, because the size of the constraints grows very fast and the small instances that can be reasonably computed give weak approximations. In some sense, our SOCP based approximation schemes may lend some of the power of semidefinite programming to the LP approximation without completely sacrificing computability.

\section{Problem-dependent inner approximation schemes}\label{sec4}

In this section, we propose some problem-dependent heuristic schemes for constructing $\{(G^k,U^k)\}$. They typically lead to computationally \revise{more tractable} problems than a positively enlarging sequence satisfying \eqref{pt_cond}. As we shall see later in our numerical experiments, these problem-dependent schemes in general return solutions with reasonable quality, though their convergence behaviors are still unknown.
\revise{A related problem-dependent approach was developed in \cite{ahmadi2017optimization} for semidefinite programming. In there, they proposed the use of the cone $\SDD^n(U)$ and progressively enlarge the $U$ to obtain efficient inner approximations to $\cS^n_+$. We propose in this section a related approach. The main difference is that in the semidefinite case considered in \cite{ahmadi2017optimization}, enlarging the $U$ is relatively simple, as we can always separate the dual solution to the inner approximation from $\cS^n_+$, if it is not there. In the case of completely positive cone, however, there is no realistic way of even checking if the dual solution is copositive. Thus, a direct separation procedure, like the one proposed in \cite{ahmadi2017optimization}, is not viable. }

\subsection{Problem-dependent positively enlarging sequence}\label{sec4.1}

In this section, we describe a problem-dependent strategy for constructing a positively enlarging sequence $\{(G^k,U^k)\}$ that can potentially perform better on specific problem instances.

After solving \eqref{SDDP0} with a choice of $(G^k,U^k)$, if the problem is feasible, one will obtain a solution $X \in \SDD_+^G(U)$. By Theorem~\ref{prop:sddGUconehull}, this $X$ can be written as a conic combination of $vv^T$ for $v \in {\rm seg}(G,U)$. Our plan here is to add these $v$ as vertices to $G$ and add some new edges from them, in order to increment the graph. The decomposition is not unique, so one has to carefully define what is meant by it.

First, note that for an $M \in \cS^2_+\cap \N^2$, there exist $a \ge 0$, $b \ge 0$ and $v\in \R^2_+$ so that
  \begin{equation}\label{Mdecomp}
  M = vv^T + \begin{bmatrix}
    a& 0\\0& b
  \end{bmatrix}.
  \end{equation}
\revise{This is trivially true if any element in the diagonal of $M$ is zero. For other matrices, the above decomposition can be realized} by taking for example $v=(\sqrt{m_{11}},m_{12}/\sqrt{m_{11}})$, implying $a=0$ and $b=m_{22}-m_{12}^2/m_{11}$, which is greater than or equal to zero since $M \succeq 0$.

Now, for any $U\in \R^{t\times n}_+$, one can see that $U^T \iota_{ij}(M) U = au_iu_i^T+bu_ju_j^T+(v_1u_i +v_2u_j)(v_1u_i +v_2u_j)^T$, where $u_i$ is the $i$th row of $U$, $1\le i<j\le t$. So, besides the vertices $u_i$ and $u_j$, we need at most one point coming from each edge $[u_i,u_j]$
to describe $U^T \iota_{ij}(M) U$. Since elements of $\SDD_+^G(U)$ are sums of matrices of this type for $\{i,j\}\in {\cal E}$ by Theorem~\ref{prop:sddGUconehull}, we have the following Lemma refining Theorem~\ref{prop:sddGUconehull}.

 \begin{lemma}\label{lem:decomposition}
  Any element $X\in \SDD_+^G(U)$ can be written as
  \[
  X=\sum_{i=1}^t \lambda_i u_i u_i^T + \sum_{\{i,j\} \in {\cal E}} \gamma_{ij} w_{ij} w_{ij}^T
  \]
  where $u_i$ is the $i$-th row of $U\in \R^{t\times n}_+$, $w_{ij} \in [u_i,u_j]$ and $\lambda_i, \gamma_{ij} \geq 0$. \rerevise{Indeed, for the first sum, it suffices to sum over the $i$'s that are covered by some edges.}
 \end{lemma}

 A natural question to ask is which points we can pick in each segment. To answer this question, we assume without loss of generality that $m_{12} > 0$ (and hence $m_{11} > 0$ and $m_{22} > 0$) in \eqref{Mdecomp} and demonstrate how the $v$ there can be chosen. Note that $U^Tvv^TU$ is supposed to correspond to a $\gamma_{ij}w_{ij}w_{ij}^T$ in the decomposition in Lemma~\ref{lem:decomposition}.

 Since $m_{12} > 0$, we must have $v_1 > 0$ and $v_2 > 0$. Then we just need to see what the ratio $r=v_1/v_2$ can be. What we saw above right after \eqref{Mdecomp} was the largest case, where we get $r=m_{11}/m_{12}$. The smallest it can get is attained by setting $v=(m_{12}/\sqrt{m_{22}},\sqrt{m_{22}})$, which gives us $r=m_{12}/m_{22}$. \revise{These two values for $r$ can be seen by noting that any extremal ratio $v_1/v_2$ for the $v$ in \eqref{Mdecomp} must correspond to $a = 0$ or $b = 0$.} 
 A balanced option, defined in a way that  the ratio between diagonal entries of $v v^T$ preserves the ratio between the diagonal entries of $M$, is to take
 \begin{equation}\label{balanced_op}
   v=\sqrt{m_{12}}\begin{bmatrix}
     \left(\frac{m_{11}}{m_{22}}\right)^{\frac14}\\
     \left(\frac{m_{22}}{m_{11}}\right)^{\frac14}
   \end{bmatrix},
 \end{equation}
 which corresponds to $r=\sqrt{m_{11}/m_{22}}$, the geometric mean of the largest and smallest possible ratios.

Based on these observations, we can now describe a general strategy for an iterative procedure to obtain upper bounds for  \eqref{P0}.

\fbox{\parbox{5.7 in}{
\begin{description}
\item {\bf Scheme 1: Successive upper bound scheme for \eqref{P0}}

\item[Step 0.] Start with a {\em complete} graph $G^0$ and its embedding $(G^0,I)$ in $\Delta^n$. Set $k=0$ and $U^0 = I$.

\item[Step 1.] For an optimal solution $X^k$ of \eqref{SDDP0} with $(G,U) =(G^k,U^k)$, apply Lemma~\ref{lem:decomposition} to obtain points $w_{ij}$ for some $\{i,j\} \in {\cal E}' \subseteq {\cal E}$
such that \rerevise{$X = X^k$} is a conic combination of $w_{ij} w_{ij}^T$ for $\{i,j\} \in {\cal E}'$ and $u_i u_i^T$ for the vertex $i$ of $G$.

\item[Step 2.] Define a new graph embedding $(G^{k+1},U^{k+1})$ by adding new vertices at the points $w_{ij}$ (or at least some subset of them) and some new edges connecting those vertices to some of the previously defined ones, and possibly remove redundant edges and go to Step 1.
\end{description}
}}
\ \\

The general idea is therefore to, \revise{augment the graph at each step by adding} some vertices in the edges that were active in the optimal solution and some edges incident with them. All the steps have, however, some subtleties that need to be addressed.

The initial embedding  $(G^0,U^0)$ is currently taken to be simply the embedding of $K^n$ into the vertices of $\Delta^n$, so that $\SDD^{G^0}_+(U^0)=\SDD^n_+$. If that is infeasible, however, the strategy does not work.
Nevertheless, assuming strict feasibility of \eqref{P0}, we know from Theorem \ref{convergdiag} that there is some small enough uniform simplicial partition of $\Delta^n$ that will make the problem feasible.

The decomposition obtained in Step 1 is not unique. There are two sources of variations. First, as discussed above, given a \revise{$2 \times 2$} semidefinite matrix $M$ such that $\iota_{ij}(M)$ \revise{appears} in the decomposition of $X$, we have some leeway on which point to pick in the edge $[u_i,u_j]$. Second, notice that even these matrices $M$ are not uniquely defined. Since the matrices $M$ will be a side result of the solution to \eqref{SDDP0}, the choice of algorithm and the way the problem is encoded will have some impact in the decomposition. As for defining the $v$ given the matrix $M$, we will use the balanced approach described above in \eqref{balanced_op} as it seems to perform well in practice.

The augmenting step (Step 2) is the most delicate of all. Different augmenting techniques will give rise to very different procedures. Here and in our numerical experiments, we consider two different approaches. We will present more implementation details in Section~\ref{sec:num}.

\paragraph{The maximalist approach:} In this approach, we add some new vertices and then connect all vertices to form a complete graph. This is \revise{memory consuming} and induces some redundancies: every node we add is in the middle of an already existing edge. Adding edges to those does not enlarge the cone $\SDD^G_+(U)$ and might lead to numerical inaccuracies, as we create multiple ways of writing points in a segment. Some pruning techniques could be applied.


\paragraph{The adaptive simplicial partition approach:} This is mimicking the technique introduced in \cite{BunDur09}, which maintains the set of edges as that of a simplicial partition. At every step we would pick edges to subdivide and subdivide all the simplices containing that edge. The choice of nodes and edges to add to $G^k$ in our approach is based on the solution we obtain from solving \eqref{SDDP0} for $(G,U) = (G^{k-1},U^{k-1})$. This is different from \cite{BunDur09}, which relies solely on an outer approximation to guide the subdivision process.\\ 

Note that we do not have any guarantee of convergence \revise{for Scheme 1}. However, geometrically one can see what must happen in order for the method to get stuck, i.e., for $\SDD^{G^k}_+(U^k)=\SDD^{G^{k+1}}_+(U^{k+1})$. As an immediate consequence of Theorem~\ref{prop:sddGUconehull}, this happens if and only if all the newly added edges in the embedding are contained in previously existing edges. \revise{This is because rank one nonnegative matrices are on the extreme rays of $\CP^n$(see \cite{abraham2003completely}). Thus, we see from Theorem~\ref{prop:sddGUconehull} that $\SDD^{G^k}_+(U^k)=\SDD^{G^{k+1}}_+(U^{k+1})$ if and only if ${\rm seg}(G^k,U^k) = {\rm seg}(G^{k+1},U^{k+1})$.} This is an extremely strong condition, that implies
essentially (depending on the scheme chosen to enlarge the graph) that the scheme gets stuck if for some iteration the optimal solution can be attained as a combination of only the nodes, and no elements from the edges. Or, in other words, the problem \eqref{SDDP0} has the same solution if we replace $\SDD^{G^k}_+(U^k)$ by \revise{$\Diag^n_+(U^k)$. On passing, we would like to point out that, in occasions where convergence is a serious concern, one can modify Step 2 of Scheme 1 by adding a random vertex in $\Delta^n$ in addition to those $w_{ij}$: this resulting scheme is guaranteed to converge in view of Theorem~\ref{thm3} if \eqref{P0} is also strictly feasible.}

\subsection{A \revise{forgetfulness scheme}}\label{sec4.2}

The use of a positively enlarging sequence $\{(G^k,U^k)\}$ can lead to large-scale SOCP problems when $k$ is huge. As a heuristic to alleviate the computational complexity, we propose a simple \revise{forgetfulness scheme}.

In this approach, we maintain the complete graph throughout. However, we always form $U^k$ by appending only the newly generated vertices to $U^0$, which we choose to be the identity matrix. The details are described below.\\

\fbox{\parbox{5.7 in}{
\begin{description}
\item {\bf Scheme 2: A \revise{forgetfulness} upper bound scheme for \eqref{P0}}

\item[Step 0.] Start with a {\em complete} graph $G^0$ and its embedding $(G^0,I)$ in $\Delta^n$. Set $k=0$ and $U^0 = I$.

\item[Step 1.] For an optimal solution $X^k$ of \eqref{SDDP0} with $(G,U) =(G^k,U^k)$, apply Lemma~\ref{lem:decomposition} to obtain points $w_{ij}$ for some $\{i,j\} \in {\cal E}' \subseteq {\cal E}$
such that \rerevise{$X = X^k$} is a conic combination of $w_{ij} w_{ij}^T$ for $\{i,j\} \in {\cal E}'$ and $u_i u_i^T$ for the vertex $i$ of $G$.

\item[Step 2.] Define a new graph embedding $(G^{k+1},U^{k+1})$: starting with $(G^0,I)$, add new vertices at the points $w_{ij}$ and then add edges between each new vertex and all vertices in $G^0$. Go to Step 1.
\end{description}
}}\ \\

Note that, in general, one cannot guarantee that the \revise{forgetfulness scheme} is even \rerevise{monotone}, as we are dropping the factors $u_i u_i^T$ that were a part of the representation of the optimal solution $X$ in Step 1. However,
in most studied random instances in our numerical experiments, the \revise{forgetfulness scheme} appears to be \rerevise{monotone}. The main reason could be that the algorithm tends to write $X$ as a conic combination of just the matrices $w_{ij} w_{ij}^T$ for $\{i,j\} \in {\cal E}'$. When this happens, we are guaranteed that the next iteration will be non-increasing, but this need not always be the case.


%


\section{Numerical simulations}\label{sec:num}

In this section, \revise{we report on numerical experiments} to test our proposed approaches. All experiments were performed in Matlab
(R2017a) on a 64-bit PC with an Intel(R) Core(TM) i7-6700 CPU (3.40GHz) and 16GB RAM.
We used the convex optimization software CVX \cite{cvx} (version 2.1), running the solver MOSEK (version 8.0.0.60) to solve the conic optimization problems that arise.

In our tests, we specifically consider the following strategies:

\paragraph{$\Delta$-partition:} In this approach, \revise{controlled by a parameter $k \geq 2$}, we generate the vertices of the graph $G^k$ as the $\binom{n+k-1}{k}$ vertices in the uniform subdivision of the simplex $\Delta^{n}$ into simplices of size $\frac1k\Delta^n$. We then add edges between two vertices whenever their supports differ by $2$.

Note that by Theorem~\ref{thm3}, if \eqref{P0} is in addition strictly feasible, then $v_p(G^k,U^k)$ will be close to $v_p$ for all sufficiently large $k$, \revise{so this strategy is guaranteed to converge as $k$ increases}.

\paragraph{Max:} This is a \revise{variant} of Scheme 1. Specifically, in Step 1, we decompose $X^k$ as described in Lemma~\ref{lem:decomposition} using the balanced option given in \eqref{balanced_op}. Then, in Step 2, \rerevise{we add to $G^k$ as new vertices all $w_{ij}$ whose corresponding entry $X^k_{ij}$} is sufficiently large as new vertices, and add edges between all vertices so that the new graph $G^{k+1}$ is complete.

\paragraph{Max1:} This is another \revise{variant} of Scheme 1. Step 1 is the same as in {\bf Max}. However, in Step 2, we {\em only} add the $w_{ij}$ corresponding to the largest $X^k_{ij}$ (if $X^k_{ij}$ exceeds a certain threshold) as a new vertex. We then add edges between all vertices so that the new graph $G^{k+1}$ is complete.

\paragraph{Adaptive $\Delta$-partition:} This is also a \revise{variant} of Scheme 1. Step 1 is the same as in {\bf Max}. For Step 2, the way of adding vertices is the same as in {\bf Max1}. However, the way we add edges mimics the approach introduced in \cite{BunDur09}, which maintains the set of edges as that of a simplicial partition. Specifically, we subdivide the edge corresponding to the $w_{ij}$ we added, and subdivide all the simplices containing that edge.

\paragraph{\revise{Forgetfulness:}} This is a \revise{variant} of Scheme 2. We perform Step 1 as in {\bf Max}. As for Step 2, we add all $w_{ij}$ whose $X^k_{ij}$ is sufficiently large as new vertices to the original graph $G^0$. We then add edges to join each newly added vertex to all vertices in $G^0$.\\

In Section~\ref{sec:ran}, we compare the strategies {\bf Max}, {\bf Adaptive $\Delta$-partition} and \revise{{\bf Forgetfulness}} on random instances of \eqref{P0}. We will also present results obtained via {\bf $\Delta$-partition} (with $k=2$) as benchmark. \rerevise{In Section~\ref{sec:stQP}, we will look at how {\bf Forgetfulness} performs on standard quadratic programs.} In Section~\ref{sec:combin}, we will first review the standard completely positive programming formulation of the stable set problem, and then examine how {\bf Max1} performs for some standard test graphs.

\subsection{Random instances}\label{sec:ran}

In order to test the performance of our method in a generic setting, we test it for randomly generated instances of problem \eqref{P0}. We generate our objective function by setting $C=M^TM$ where $M$ is an \revise{$n \times n$} matrix with i.i.d. standard Gaussian entries, guaranteeing strict feasibility of \eqref{P0dual}. Furthermore, we generate the constraints by setting $A_i=(M_i+M_i^T)/2$, where the $M_i$ are also \revise{$n \times n$} matrices with i.i.d. standard Gaussian entries, and choosing $b_i$ such that $b_i={\rm tr}(A_i(E+nI))$. This guarantees strict feasibility of \eqref{P0}.

For the first of our tests we varied the number of variables, $n$, and the number of constraints $m$, so that $n$ is either $10$ or $25$ and $m$ is either $5$, $10$ or $15$. Given the complexity of copositive programming, there is actually no reliable way to find the true solution for these problems and there is no available implemented method that can generate lower bounds with which to compare our results. As a work-around, throughout this section we will compare the results we obtain with the classical (and somewhat coarse) lower bound provided by replacing $\CP^n$ by $\cS_+^n\cap \N^n$ in problem \eqref{P0}. We will use the difference of our approximations to this lower bound, normalized by \revise{dividing it by the bound}, as a proxy for the quality of the methods, and will simply denote it by \revise{relative} gap. \revise{Precisely, this quantity is defined by $\textup{gap}(x)=\frac{x-x^*}{|x^*|}$, where $x$ is the objective value attained by the method being studied and $x^*$ the doubly nonnegative lower bound.} This makes it somewhat easier to compare different methods across different instances of the problem. \revise{The drawback is that the gap we compute is actually the sum of the gaps of the proposed method and the doubly nonnegative approximation, which we don't know how to independently estimate.}

\begin{table}[h]
	
	\begin{center}
		\begin{tabular}{|c|c|c|c|c|c|c|c|c|c|}
			\hline
		\multicolumn{2}{|c|}{} &	\multicolumn{2}{|c|}{Max} & \multicolumn{2}{|c|}{Adaptive $\Delta$-Partition} & \multicolumn{2}{|c|}{Forgetfulness}& \multicolumn{2}{|c|}{$\Delta$-Partition}\\\hline
		     n  &   m  &
			time\revise{(sec)} & Relative Gap &
			time\revise{(sec)} & Relative Gap &
			time\revise{(sec)} & Relative Gap &
			time\revise{(sec)} & Relative Gap  \\ \hline

10 &  5  &   8.2 & 5.035e-02 &  25.3 & 6.362e-02 &   13.0 & 2.006e-02 &   4.6 & 4.620e-01 \\
10 & 10  &  19.5 & 2.281e-02 &  25.3 & 7.920e-02 &   23.0 & 1.849e-02 &   4.5 & 4.095e-01 \\
10 & 15  &  41.7 & 1.212e-02 &  27.6 & 8.207e-02 &   27.0 & 1.179e-02 &   5.0 & 2.995e-01 \\

25 &  5 &  23.0 & 6.748e-01 &  55.1 & 5.828e-01  &  38.4 & 2.975e-01  &  --- & --- \\	
25 & 10 &  45.8 & 4.660e-01 &  62.9 & 7.841e-01  &  52.7 & 2.020e-01  &  --- & --- \\
25 & 15 &  71.8 & 3.715e-01 &  56.1 & 8.565e-01  &  61.5 & 1.545e-01  &  --- & --- \\	

			\hline
		\end{tabular}
	\caption{Comparison of different iterative approaches\label{tab1}}
	\end{center}	
\end{table}

The results obtained can be seen in \revise{Table~\ref{tab1}}, where we present both the average gaps and the average running time for the studied methods. A few technical details are needed to be able to replicate the experiment. The results presented are averages of 30 instances per parameter pair. Moreover we fix the maximum number of iterations for the Max, Adaptive $\Delta$-partition and \revise{Forgetfulness schemes} as, respectively, $5$, $20$ and $15$ for $n=10$ and $5$, $15$ and $12$ for $n=25$. This was done (in an ad hoc way) to try to keep the average execution time as similar as possible across iterative methods, so that a fair comparison can be made. Also, since the maximalist approach can occasionally explode in size, we also stop this approach early when $t_{k+1}> 200$ (Recall that $U^k\in \R^{t_k\times n}_+$ for all $k$). For the \revise{forgetfulness} approach, we prune the $U^k$ in each step by removing redundant rows: we compute $\delta^k_{ij} := \|u^k_i-u^k_j\|_1$, where $u^k_i$ and $u^k_j$ are the $i$-th and the $j$-th rows of $U^k$ respectively, $j > i$, and discard $u^k_j$ if $\delta^k_{ij}< 10^{-6}$. We also stop this approach early when $t_{k+1}> 200$ for the $U^{k+1}$ after pruning.
The static $\Delta$-partition is not computed for $n=25$ as it \revise{takes too long}. 

These results show that the \revise{Forgetfulness scheme} dominates the others in all categories as far as the relation quality/time is concerned. The \revise{relative gaps} of the attained solutions jumps from \revise{between $1\%$ and $2\%$ for $n=10$ to between $15\%$ and $30\%$} for $n=25$. Once again, we stress that these are upper bounds for the \revise{Forgetfulness scheme} quality as well as for the doubly nonnegative approximation quality, and we cannot separate the contributions from each method.

We also plot in Figure~\ref{fig:fgtgap} the evolutions of the gaps for the \revise{Forgetfulness scheme} for $10$ random instances of the problem \eqref{P0} with $n=25$ and $m=10$. We can see the logarithmic scale plot of the gap as iterations increase, and the diminishing returns in improvement percentage. Again, note that the true gap might actually be decreasing faster, as what we are seeing is the gap to the doubly nonnegative lower bound.

\begin{figure}

\includegraphics[width=\linewidth]{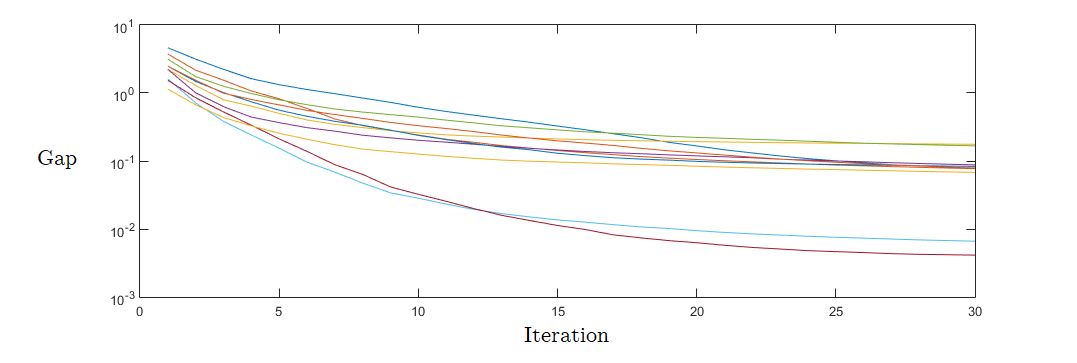}

\caption{Evolution of the gap for the \revise{forgetfulness scheme} as iterations increase}
\label{fig:fgtgap}
\end{figure}


\subsection{Standard Quadratic Program}\label{sec:stQP}

We now focus on a class of more structured completely positive programs, those coming from standard quadratic programs. A standard quadratic optimization problem (SQP) consists of finding global minimizers for a quadratic form over the standard simplex. In other words, given $p(x)=x^TQx$ for some $Q \in {\cal S}^n$, we want to find its minimum over the simplex $\Delta=\{x \in \R^n_+ \ : \ \sum x_i = 1\}$. It is shown in \cite{Bom00} that this can be written as the completely positive program
\begin{equation}\label{SQP}
  \begin{array}{rl}
    p^*:=\min & {\rm tr}(QX)\\
    {\rm s.t.} & {\rm tr}(E X) = 1,\\
    & X \in \CP^n,
  \end{array}
\end{equation}
where $E$ is the all ones matrix. It is not difficult to see that these problems always verify the blanket assumptions presented in Section \ref{sec1.1}. Furthermore, since
$\frac{1}{n}I_n$ is feasible, our hierarchy can always start from the base $\SDD$ relaxation.

To illustrate the behaviour of our method we start by taking the four concrete examples collected in \cite{Bom02} from several domains of application and applying the Forgetfulness scheme. We get the encouraging results shown in Table~\ref{Table:exSQP}, where we can see the source of the examples, their size $n$, their true solution $p^*$, and the approximate solution obtained by the Forgetfulness scheme in $5$ iterations (reported under the column approx.). We can see that the third example is the only one where there is a significant deviation from the optimal value, and all the results were attained in a few seconds.

\begin{table}[h]
\begin{center}
\begin{tabular}{|l|c|c|c|c|} \hline
Example                                                          & $n$  &  $p^*$  &  approx.    \\ \hline
\cite[Example~5.1]{Bom02} - independence number of pentagon                & $5$  &  $1/2$ & $0.50000$              \\ \hline
\cite[Example~5.2]{Bom02} - independence number of icosahedron complement  & $12$  & $ 1/3$  & $0.33333$              \\ \hline
\cite[Example~5.3]{Bom02} - math. model of population genetics  & $5$  &  $-16\frac{1}{3}$  &  $-16.331$              \\ \hline
\cite[Example~5.4]{Bom02} - portfolio optimization  & $5$  &  $0.4839$  & $0.4839$               \\ \hline
\end{tabular}
\caption{Applying the Forgetfulness scheme in four small SQP examples} \label{Table:exSQP}
\end{center}
\end{table}

To further explore the behaviour of our approach we followed the idea of \cite{Bom02} to generate random instances of SQP. In that paper they generate matrices $Q$ to be $10 \times 10$, symmetric and with entries uniformly distributed in the interval $[0,1]$. However, solving five thousand random examples of such problems to global optimality with CPLEX, we noticed that the true solutions seem to be commonly in the vertices of the simplex ($48.5\%$ of observed instances) or in edges ($40.1\%$ of observed instances). But in those two cases, by Theorem~\ref{prop:sddUconehull} our relaxation finds the optimal value at the first step. In other words, simply replacing $\CP^{10}$ by $\SDD_*^{10}$ gives us the exact solution in $88.6\%$ of the times, with our iterative procedure only kicking in in the remaining instances.

To get a more meaningful test, we generated symmetric matrices $Q$ with diagonal $1$ and only off-diagonal entries uniformly distributed in the interval $[0,1]$. Experimentally this virtually never gives rise to optimal solutions in the edges of the simplex, leading to non-trivial instances.
We tested for both $n=10$ and $n=15$, comparing the results against the true value obtained using CPLEX.
The parameters were chosen in the same way as in the previous section with the number of iterations being $15$ for $n=10$ and $12$ for $n=15$.

\begin{figure}
\begin{center}
\includegraphics[width=0.45\linewidth]{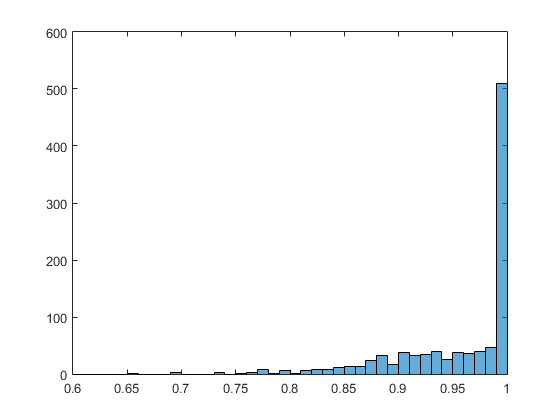} \hspace{1cm} \includegraphics[width=0.45\linewidth]{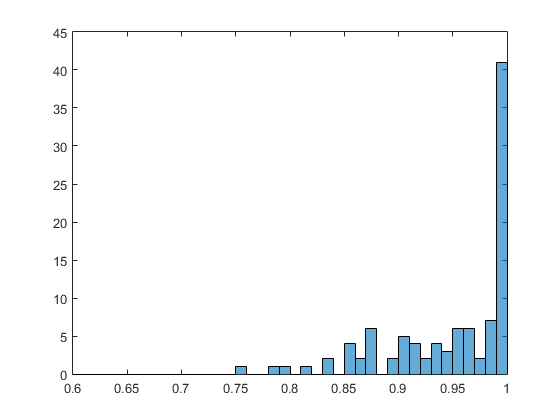} \\
\includegraphics[width=0.45\linewidth]{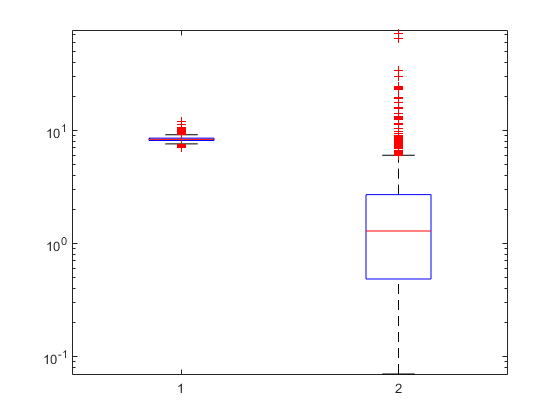} \hspace{1cm} \includegraphics[width=0.45\linewidth]{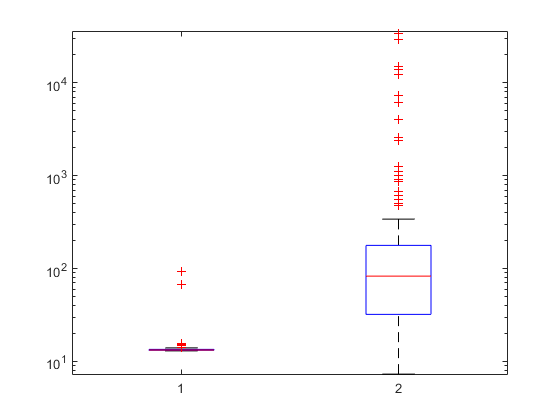} \\
\caption{Results of the random SQP tests for $n=10$ (left) and $n=15$ (right). The top graphs are histograms of the ratio between the true value and the attained upper bound, the bottom graphs box plots of the CPU times (in seconds) used by our method (1) and CPLEX (2).} \label{fig:SQPresults}
\end{center}
\end{figure}

We ran $1000$ instances for $n=10$ and $100$ for $n=15$. The results are presented in Figure \ref{fig:SQPresults}. On the top row we show the histograms for the ratios between the true value, computed using CPLEX, and our computed approximation, which provides an upper bound. We can see among other things that in both cases around half the instances were within $1\%$ of the true value and four fifths were within $10\%$. If we want to more directly compare it with the results attained for the random instances in Table~\ref{tab1}, one can compute the mean value of the true relative gap $\frac{\hat p-p^*}{p^*}$ where $p^*$ is the true optimal value returned by CPLEX and $\hat p$ is the approximate value obtained by our approach. We get $4.823\times 10^{-2}$ and $5.747\times 10^{-2}$ for $n=10$ and $n=15$ respectively, very much in line to what we have seen before.

On the bottom row of Figure~\ref{fig:SQPresults}, as a rough reference, we have the boxplots of the CPU times (in seconds) taken by our method and CPLEX, presented here in logarithmic scale for readability. In both cases we can see that the Forgetfulness scheme is quite stable, as it will simply stop after a set number of iterations, while CPLEX has a huge number of outliers. While for $n=10$ the exact CPLEX computation is faster, in $n=15$ it becomes much slower, with several outliers taking many hours. For larger values of $n$ it quickly becomes prohibitively slow compared to our approach.

\subsection{Stable set problems}\label{sec:combin}

While in the previous section we focus on random problems, the main focus of the completely positive/copositive programming literature has been in highly structured combinatorial optimization problems.
One of the most common applications is to the stable set problem, i.e., the problem of finding in a graph $G$ a  set of vertices of maximal cardinality such that no two are connected with an edge. The cardinality
\revise{of such a set is known as} the independence number of $G$, denoted by $\alpha(G)$.
In \cite[Equation~(8)]{de2002approximation}, the following completely positive formulation was introduced for that problem.
\begin{equation}\label{stableset}
\begin{array}{rl}
{\alpha(G)}=\max & {\rm tr}(EX)\\
{\rm s.t.} & {\rm tr}((A_G+I)X) =1,\\
& X \in \CP^n,
\end{array}
\end{equation}	
where $A_G$ is the adjacency matrix of $G$.

In this setting we have \revise{a single constraint}, so $m=1$. Our inner approximations of $\CP^n$ will yield in this case lower bounds, from which one might be able to extract an actual feasible stable set with given cardinality. There are a number of good heuristic approaches to the stable set problem with good results, as there exist implementations of exact algorithms that can handle small to medium sized graphs, all performing necessarily much better than our all-purpose conic programming approach. However, we can still see how our approach performs on its own, to get some indication of its performance on low codimension structured problems.

In this class of problems, symmetry and structure likely imply that the growth of the matrix $U$ in  the greedier Maximalist approach but also in the Forgetfulness approach is too fast and adds too much redundancy.
To avoid this phenomenon we take the Max1 approach: at every iteration we only add to $U$ the vertex that has the largest weight in the solution found. This yields \revise{a greedy sort of algorithm}, that in practice tends to grow the stable set greedily one by one. We stopped as soon as the greedy process got stuck and there was no improvement in two consecutive iterations.

We computed both stability numbers, $\alpha(G)$, and clique numbers, \revise{$\omega(G)$}, which are simply the stability numbers of the complementary graph. Following \cite{BunDur09}, we started by computing the clique numbers of the graphs where their method was tested.
\revise{Our method yields the correct answers in a relatively short time, as can be seen in Table~\ref{tabclique}, where our results are presented under the column ``result", and the column ``\revise{$\omega(G)$}" corresponds to the known clique numbers. Note that this is not too surprising, as finding a large stable set, or clique, is in a general sense computationally easier than proving that a larger one does not exist. In other words, lower bounding the stable set and clique numbers of particular graphs tends to be easier than upper bounding them, so our problem has a smaller scope than what was attempted in \cite{BunDur09}, leading to much faster times.}
The graphs in the table come from two sources, the first is a $17$ vertex graph from \cite{Pen07} that is notoriously hard for upper bounding by convex approximations, the other five come from the 2nd DIMACS implementation challenge test instances \cite{DIMACS}, and only hamming6-4 and johnson8-2-4 could be solved by Bundfuss and D\"{u}r's method in less than two hours as reported in their paper \cite{BunDur09}.

\begin{table}[h]
	\begin{center}
	
		\begin{tabular}{|c | c |c |c | c|c |}	\hline
			graph & vertices & iterations & time(sec) &  result & \revise{$\omega(G)$}\\	\hline
			
		    pena17        &   17 &  5 &  13.8 &   6.0000&   6\\
		
			hamming6-2    &  64 &  31 & 836.7 &  32.0000 &  32\\
			
			hamming6-4    &   64&   3 &  64.0 &   4.0000 &   4\\
			
			johnson8-2-4  &  28 &   3 &  11.7 &   4.0000 &   4\\
			
			johnson8-4-4  &  70 &  13 & 322.5 &  14.0000  &  14\\
			
			johnson16-2-4 &  120 &  7 & 637.0 &   8.0000 &   8\\		
			
			\hline
		\end{tabular}
		\caption{Clique number for different graphs}
	\label{tabclique}
	\end{center}
	
\end{table}

To explore the limits of our approach we tried a few more instances of the stable set problem. We tried Paley graphs, known to mimic some properties of random graphs, with some degree of success, and a few small-sized instances of graphs derived from error correcting codes, available at \cite{slo05}. The results are much worse in this family, with our algorithm failing in small instances, as can be seen in Table~\ref{tabstab}, where our results are reported under the column ``result", and the true stability numbers are presented under the column ``$\alpha(G)$". One word of caution is that the entire procedure is highly unstable, and simply changing the solver from MOSEK to SDPT3 can result in changes in the result, e.g. Paley$_{137}$ becomes exact in SDPT3.

\begin{table}[h]
	\begin{center}

		\begin{tabular}{|c | c |c |c | c|c|}	\hline
			graph & vertices & iterations & time(sec) & result & $\alpha(G)$ \\	\hline
			Paley$_{137}$ & 137 & 4 & 977.4 &  5.0000  &  7\\
			Paley$_{149}$ & 149 & 6 & 1841.6 &  7.0000 &  7\\
			Paley$_{157}$ & 157 &   6 & 2254.1 &  7.0000 &  7\\			
			1tc.16 & 16 &    6 &  15.7 &   7.0000 & 8\\			
			1tc.32 &  32 &  10 &  85.5 &  11.0000 & 12\\			
			1dc.64 &  64 &   7 & 235.8 &   8.0000 & 10\\			
			1dc.128 &  128 &  13 & 2491.0 &  14.0000 & 16\\
			2dc.128 &  128 &   4 & 823.6 &   5.0000 & 5 \\
		\hline
		\end{tabular}
	\caption{Stability number for different graphs}
	\label{tabstab}

	\end{center}
	
\end{table}


\bibliographystyle{plain}
\bibliography{ref}

\end{document}